\documentclass[11pt]{article}

\usepackage[utf8]{inputenc}
\usepackage[english]{babel}
\usepackage[a4paper, margin=2.5cm]{geometry}
\usepackage{amsmath, amsthm, amssymb, amsfonts}
\usepackage{mathtools}
\usepackage[shortlabels]{enumitem}
\usepackage{hyperref}
\usepackage[font={small, it}]{caption}
\usepackage{subcaption}
\usepackage{graphicx}
\usepackage{color}
\usepackage{xcolor}
\usepackage{thmtools}
\usepackage{bm}
\usepackage{thm-restate}

\newcommand{\cE}{\ensuremath{\mathcal E}}

\newcommand{\cP}{\ensuremath{\mathcal P}}

\newcommand{\ATYP}{\ensuremath{\mathrm{ATYP}}}
\newcommand{\TINY}{\ensuremath{\mathrm{TINY}}}

\newcommand{\eps}{\varepsilon}
\renewcommand{\phi}{\varphi}

\DeclareMathOperator*{\E}{\mathbb{E}}
\DeclareMathOperator*{\N}{\mathbb{N}}

\DeclarePairedDelimiter\floor{\lfloor}{\rfloor}

\newcommand{\Gnp}{G_{n, p}}

\newcommand{\Gnm}{G_{n, m}}

\definecolor{royalazure}{rgb}{0.0, 0.22, 0.66}

\declaretheorem[parent=section]{theorem}
\declaretheorem[sibling=theorem]{lemma}
\declaretheorem[sibling=theorem]{proposition}
\declaretheorem[sibling=theorem]{claim}

\declaretheorem[sibling=theorem,style=definition]{definition}
\declaretheorem[parent=section,style=remark]{remark}

\setlist{itemsep=0.1em, topsep=0.1em}

\makeatletter
\def\thm@space@setup{%
  \thm@preskip=0.5em \thm@postskip=0.2em
}
\makeatother

\hypersetup{
    colorlinks,
    linkcolor={red!60!black},
    citecolor={green!50!black},
    urlcolor={blue!80!black}
}

\title{Resilience of Perfect Matchings and Hamiltonicity in Random Graph Processes}
\author{
  Rajko Nenadov\thanks{School of Mathematical Sciences, Monash University, VIC 3800, Australia \newline Email: \texttt{rajko.nenadov@monash.edu}}
  \and
  Angelika Steger\thanks{Institute of Theoretical Computer Science, ETH Z\"{u}rich, 8092 Z\"{u}rich, Switzerland \newline Email: \{\texttt{steger}\textbar \texttt{mtrujic}\}\texttt{@inf.ethz.ch}}
  \and
  Milo\v{s} Truji\'{c}\footnotemark[2] \textsuperscript{,}\thanks{author was supported by grant no.\ 200021 169242 of the Swiss National Science Foundation.}
}
\date{}

\begin{document}
\maketitle

\begin{abstract}
  Let $\{G_i\}$ be the random graph process: starting with an empty graph $G_0$
  with $n$ vertices, in every step $i \geq 1$ the graph $G_i$ is formed by
  taking an edge chosen uniformly at random among the non-existing ones and
  adding it to the graph $G_{i - 1}$. The classical `hitting-time' result of
  Ajtai, Koml\'{o}s, and Szemer\'{e}di, and independently Bollob\'{a}s, states
  that asymptotically almost surely the graph becomes Hamiltonian as soon as the
  minimum degree reaches~$2$, that is if $\delta(G_i) \ge 2$ then $G_i$ is
  Hamiltonian. We establish a resilience version of this result.  In particular,
  we show that the random graph process almost surely creates a sequence of
  graphs such that for $m \geq (\tfrac{1}{6} + o(1))n\log n$ edges, the $2$-core
  of the graph $G_m$ remains Hamiltonian even after an adversary removes
  $(\tfrac{1}{2} - o(1))$-fraction of the edges incident to every vertex. A
  similar result is obtained for perfect matchings.
\end{abstract}

%!TEX root = threshold_resilience.tex
\section{Introduction}

The theory of random graphs originated in 1959 with the two seminal papers by Erd\H{o}s and R\'enyi~\cite{erdds1959random} and Gilbert~\cite{gilbert1959random}. It is now a well-established research area with many applications in theoretical computer science, statistical physics, and other branches of mathematics, cf.~\cite{bollobas1998random, FK15, JLR11}. Typical questions in random graph theory traditionally concern the \emph{existence} of certain (sub)structures. A more recent trend, introduced by Sudakov and Vu~\cite{SV08}, is the study of their \emph{resilience} properties. Formally, resilience of a graph~$G$ with respect to some property~${\cal P}$ is defined as follows:
\begin{definition}
  Let $G = (V, E)$ be a graph. We say that $G$ is \emph{$\alpha$-resilient} with respect to the property $\mathcal{P}$, for some $\alpha \in [0,1]$, if for every spanning subgraph $H \subseteq G$ such that $\deg_H(v) \le \alpha \deg_G(v)$ for every $v \in V$, we have $G - H \in \mathcal{P}$.
\end{definition}

In other words, $\alpha$-resilience implies that the property~${\cal P}$ cannot be destroyed even if an adversary is allowed to remove an (arbitrary) $\alpha$-fraction of all edges incident to each vertex. For instance, Dirac's celebrated theorem~\cite{Dir52} states that $K_n$ is $(1/2)$-resilient with respect to containing a Hamilton cycle. The study of resilience of complete graphs is one of the central topics in extremal combinatorics (see, e.g.~\cite{bottcher2009proof, hajnal1970proof, komlos1998posa, kuhn2009minimum} for some of the cornerstones).

In this paper we show that two seminal results of random graph theory hold in a resilient fashion. More precisely, we consider two famous `hitting-time' results---one by Erd\H{o}s and R\'enyi~\cite{erdos66matching} and Bollob\'as and Thomason \cite{bollobas1985random} regarding perfect matchings, and one by Ajtai, Koml\'os, and Szemer\'edi~\cite{ajtai1985first} and, independently, Bollob\'as~\cite{bollobas84evolutionhitting} regarding Hamiltonicity. Let $n \in \mathbb{N}$ and consider the following process: set $G_0$ to be an empty graph with $n$ vertices and for each $i \in \{1, \ldots, \binom{n}{2}\}$ form the graph $G_i$ by choosing an edge $e_i \notin G_{i-1}$ uniformly at random and adding it to $G_{i-1}$. In this way we obtain a sequence of nested graphs $\{G_i\}_{i = 0}^N$ for $N = \binom{n}{2}$, where $G_0$ is an empty graph and $G_N$ is a complete graph. It is an easy exercise to show that for every $m \in \{1, \ldots, N\}$ the graph $G_m$ has the same distribution as the Erd\H{o}s-R\'enyi random graph $G_{n,m}$, a graph chosen uniformly at random among all labelled graphs with $n$ vertices and exactly $m$ edges. A necessary condition for a graph to contain a perfect matching is that the minimum degree is at least 1 (assuming that $n$ is even). Similarly, having minimum degree at least $2$ is a necessary condition to be Hamiltonian. The aforementioned results show that, perhaps surprisingly, these conditions are also sufficient: as soon as $\delta(G_m) \ge 1$ the graph $G_m$ contains a perfect matching, and as soon as $\delta(G_m) \ge 2$ the graph contains a Hamilton cycle, both {asymptotically almost surely}\footnote{An event is said to hold asymptotically almost surely (a.a.s.\ for short) if the probability that it holds approaches $1$ as $n \to \infty$.}. This kind of results are usually referred to as \emph{hitting-time} results.

In this paper we show that, in fact, as soon as the minimum degree reaches $1$ the graph does not only contain a perfect matching, but it does so {\em robustly}---meaning it is $(1/2-o(1))$-resilient with respect to this property. Similarly, as soon as the minimum degree reaches $2$ the graph $G_m$ becomes robustly Hamiltonian, that is $(1/2-o(1))$-resilient with respect to being Hamiltonian. To see why the parameter $1/2$ is asymptotically best possible, first recall that every graph $G$ contains a partition $V_1 \cup V_2$ such that every vertex in one part has at least half of its neighbours in the other. In the case when such a partition is not balanced, deleting all the edges within $V_1$ and $V_2$ clearly prevents the resulting graph from having either a perfect matching or a Hamilton cycle. Otherwise, both parts have size $n/2$ and it is an easy exercise to show that if $G$ is a random graph with $m \gg n$ edges, which will be the case here, at least one vertex in $V_1$ also has roughly half of its neighbours in $V_1$. Moving such a vertex to $V_2$ results in an unbalanced partition, and the rest of the argument is the same as before.

Even though the study of resilience properties of random graphs has attracted considerable attention over the last years, cf.\ e.g.~\cite{balogh2011local,bottcher2013almost,huang2012bandwidth,krivelevich2010resilient} and the recent survey~\cite{sudakov2017robustness}, this is the first result which determines precisely when the random graph process becomes resilient with respect to some well-studied property.

\subsection{Results}
The classical hitting-time result for perfect matchings by Bollob\'as and Thomason \cite{bollobas1985random}, and a previously obtained result by Erd\H{o}s and R\'enyi~\cite{erdos66matching} which establishes the correspondence between minimum degree one and perfect matchings in a slightly weaker form, follow quite easily from well-known sufficient conditions for the existence of perfect matchings. On the other hand, the hitting time result for Hamiltonicity by Ajtai, Koml\'os, and Szemer\'edi~\cite{ajtai1985first} and Bollob\'as~\cite{bollobas84evolutionhitting} is significantly more intricate and builds on a result of P\'osa~\cite{posa1976hamiltonian} and subsequent refinements.

The situation is similar for our results. While Hamiltonicity is more interesting and challenging (and our main result), we use the case of perfect matchings to demonstrate some of our main ideas and the general approach.

\paragraph{Perfect matchings.}
Our first result shows a bit more than discussed earlier: as long as the number of added edges is not too small, removing isolated vertices, an obvious obstacle in obtaining a perfect matching, results in a graph which is resilient with respect to containing a perfect matching\footnote{In order to avoid having to distinguish between an odd and an even number of vertices of the given graph, we use the term `perfect matching' to indicate that the graph contains a matching that covers all but at most one of its vertices.}.
\begin{theorem} \label{thm:main-thm-pm}
  Let $\eps > 0$ be a constant and consider the random graph process $\{G_i\}$. Then a.a.s.\ for every $m \ge \frac{1 + \eps}{4} n \log n$ we have that the graph obtained from $G_m$ by deleting all isolated vertices is $(1/2 - \eps)$-resilient with respect to containing a perfect matching.
\end{theorem}

This improves a result of Sudakov and Vu~\cite{SV08} who showed that a random graph $G_{n,m}$ is $(1/2 - \eps)$-resilient with respect to containing a perfect matching if $m \geq C n \log n$, for some large constant $C(\eps) > 100$. Note that the bound on $m$ in Theorem~\ref{thm:main-thm-pm} is asymptotically optimal. The reason being that for $m = \frac{1 - \eps}{4} n \log n$ there exist many \emph{cherries} (see, e.g.~\cite{bollobas1985random}), pairs of vertices of degree $1$ which have a common neighbour. Clearly, for each cherry only one vertex can be part of a matching and a perfect matching thus cannot exist. The standard non-resilience version of Theorem \ref{thm:main-thm-pm} with somewhat more precise lower bound on $m$ was established in~\cite{bollobas1985random}.

In fact, we obtain a slightly stronger resilience statement. That is, even if we allow the adversary to delete a graph $H$ that contains almost all the edges incident to `atypical' vertices (vertices whose degree deviates significantly from the average degree), the resulting graph still contains a perfect matching. The precise statement can be found in Section \ref{sec:matching}.

Theorem \ref{thm:main-thm-pm} also immediately implies a resilience version of the hitting-time result for perfect matchings.
\begin{theorem}
  Let $\eps > 0$ be a constant. Consider the random graph process $\{G_i\}$ with an even number of vertices $n$ and let $m_1 = \min \{m \colon \delta(G_m) = 1\}$ denote the step in which the last isolated vertex disappears. Then a.a.s.\ we have that $G_{m_1}$ is $(1/2 - \eps)$-resilient with respect to containing a perfect matching.
\end{theorem}

\paragraph{Hamiltonicity.}
Prior to proving that the point when $\delta(G_m)$ becomes $2$ coincides with the point when $G_m$ is Hamiltonian, Koml\'os and Szemer\'edi \cite{komlos1983limit} showed that if $m = \frac{n}{2} (\log n + \log\log n + c_n)$ then $G_{n,m}$ contains a Hamilton cycle if $c_n \to \infty$. On the other hand, if $c_n \to -\infty$ then it is a.a.s.\ not Hamiltonian (see \cite{komlos1983limit}), precisely because of the existence of a vertex with degree less than $2$. However, something can still be said for such a sparse graph, or rather its \emph{$2$-core}. Given a graph $G$ we define the \emph{$2$-core} of $G$ to be the graph obtained from $G$ by successively removing vertices of degree at most $1$. In other words, the $2$-core of a graph $G$ is its largest subgraph which satisfies the necessary minimum degree condition for being Hamiltonian. \L uczak~\cite{luczak1987matchings} showed that, in the case of a random graph $G_{n, m}$, such subgraphs are indeed Hamiltonian as long as $m \ge \tfrac{1 + \eps}{6} n\log n$. Here we extend this by showing that, starting from that point, the $2$-core of every graph $G_m$ is resilient with respect to containing a Hamilton cycle. Moreover, the bound on $m$ is asymptotically optimal, as a.a.s.\ the $2$-core of $G_{n, m}$ is not Hamiltonian for $m \le \tfrac{1 - \eps}{6} n \log n$ (see~\cite{luczak1987matchings}).
\begin{theorem} \label{thm:main-thm-ham}
  Let $\eps > 0$ be a constant and consider the random graph process $\{G_i\}$. Then a.a.s.\ for every $m \ge \frac{1 + \eps}{6} n \log n$ we have that the $2$-core of $G_m$ is $(1/2 - \eps)$-resilient with respect to being Hamiltonian.
\end{theorem}

Similarly to the case of perfect matchings we actually prove a slightly stronger resilience result that allows deletion of almost all edges incident to `atypical' vertices, cf.\ Section~\ref{sec:HAM} and Theorem~\ref{thm:stronger-resilience-ham} for details.

The question of resilience of random graphs with respect to Hamiltonicity has attracted considerable attention in the last years. Sudakov and Vu \cite{SV08} showed that $G_{n,m}$ is $(1/2 - o(1))$-resilient with respect to Hamiltonicity for $m \ge n \log^4 n$, and conjectured that already $m \gg n \log n$ should suffice. The lower bound on $m$ was gradually improved in a series of paper: Frieze and Krivelevich \cite{frieze2008two} showed that $\Gnm$ is $\alpha$-resilient for $p \ge K n \log n$, for some small $\alpha > 0$ and sufficiently large $K$; Ben-Shimon, Krivelevich, and Sudakov~\cite{BKS11} showed $(1/6 - o(1))$-resilience and soon after improved it to $(1/3 - o(1))$-resilience~\cite{ben2011resilience} for $m \ge \tfrac{1 + \eps}{2}n \log n$; finally, Lee and Sudakov \cite{LS12} showed the optimal $(1/2 - o(1))$-resilience for $m \ge K n \log n$.

While the last result settles the original conjecture of Sudakov and Vu, the gap between the value of $m$ for which $\Gnm$ is typically Hamiltonian and for which we know is $(1/2 - o(1))$-resilient with respect to Hamiltonicity remains. A simple corollary of Theorem \ref{thm:main-thm-ham} closes this gap by showing that the two coincide. More precisely, we obtain that as soon as $\delta(G_m) \ge 2$ the graph $G_m$ is resilient with respect to Hamiltonicity. This provides a resilience version of the hitting-time result of Ajtai, Koml\'os, and Szemer\'edi~\cite{ajtai1985first} and Bollob\'as~\cite{bollobas84evolutionhitting}.
\begin{theorem} \label{thm:hitting_time_HAM}
  Let $\eps > 0$ be a constant. Consider the random graph process $\{G_i\}$ and let $m_2 = \min \{m \colon \delta(G_m) = 2\}$ denote the step in which the last vertex of degree at most one disappears. Then a.a.s.\ we have that $G_{m_2}$ is $(1/2 - \eps)$-resilient with respect to being Hamiltonian.
\end{theorem}

We note that Theorem~\ref{thm:hitting_time_HAM} (but not~\ref{thm:main-thm-ham}) was proven independently also by Montgomery~\cite{montgomery2017} using a different approach.

\paragraph{Structure of the paper.}
In the next section we give some necessary definitions and tools used throughout the paper. In addition to providing standard facts, we introduce the notion of \emph{tiny} and \emph{atypical} vertices and a strengthened notion of $\alpha$-resilience. In Section \ref{sec:matching} we prove (a strengthened version) of Theorem \ref{thm:main-thm-pm} that is based on this new version of $\alpha$-resilience. Even though the proof strategy is simple on a high level, it requires many intricate details related to tiny and atypical vertices. In Section~\ref{sec:HAM} we use some of these ideas and combine them with the P\'osa's \emph{rotation-extension} technique to prove a strengthened version of Theorem~\ref{thm:main-thm-ham}.

%!TEX root = threshold_resilience.tex
\section{Preliminaries}

Our graph theoretic notation is standard and follows the one from \cite{diestel2000graph}. In particular, given a graph $G$ and (not necessarily disjoint) subsets $X, Y \subseteq V(G)$, we denote by $e_G(X, Y)$ the number of edges of $G$ with one endpoint in $X$ and the other in $Y$. Note that every edge which lies in the intersection of $X$ and $Y$ is counted twice. We denote by $e_G(X)$ the number of edges with both endpoints in $X$. Furthermore, given a vertex $v$ and an integer $\ell \geq 1$, we denote by $N_G(v)$ its set of neighbours and by $N_{G}^{\ell}(v)$ the set of all vertices which are at distance at most $\ell$ from $v$, excluding $v$. Given graphs $G$ and $H$ with the same vertex set, we define $G - H$ as the graph with vertex set $V(G)$ and edge set $E(G) \setminus E(H)$. Moreover, for a graph $G$ and a set of edges $E$ on the same vertex set, we write $G + E$ to denote the graph on the vertex set $V(G)$ and edge set $E(G) \cup E$. For a positive integer $n$ and a function $0 \leq p := p(n) \leq 1$ we let $G_{n, p}$ denote the probability space of graphs with vertex set $[n] = \{1, \ldots, n\}$ where each pair of vertices forms an edge of $G \sim G_{n, p}$ with probability $p$, independently of all other pairs. We make use of the standard asymptotic notation $o, O, \omega$, and $\Omega$. In addition, we write $O_{\eps}$ or $\Omega_{\eps}$ to emphasise that the hidden constant depends on a parameter $\eps$. For two functions $a$ and $b$ we write $a \ll b$ to indicate $a = o(b)$ and $a \gg b$ to indicate $a = \omega(b)$. In all occurrences $\log$ denotes the natural logarithm. We mostly suppress floors and ceilings whenever they are not crucial. Given $\eps, x, y \in \mathbb{R}$, we write $x \in (1 \pm \eps)y$ to denote $(1 - \eps)y \le x \le (1 + \eps)y$. Finally, we use subscripts with constants such as $C_{2.5}$ to indicate that $C_{2.5}$ is a constant given by Claim/Lemma/Proposition/Theorem 2.5.

Throughout the paper we make use of the following standard estimate on tail probabilities of a binomial random variable $\text{Bin}(n, p)$ with parameters $n$ and $p$, see e.g.~\cite{FK15}.
\begin{lemma}[Chernoff bounds]\label{lem:chernoff}
  Let $X \sim \text{Bin}(n, p)$ and let $\mu := \E[X]$. Then for all $0 < \delta < 1$:
  \begin{itemize}
    \item $\Pr[X \geq (1 + \delta)\mu] \leq e^{-\frac{\delta^2\mu}{3}}$, and
    \item $\Pr[X \leq (1 - \delta)\mu] \leq e^{-\frac{\delta^2\mu}{2}}$.
  \end{itemize}
\end{lemma}
\begin{remark}
  This result remains true if $X$ has a hypergeometric distribution, cf.~\cite[Theorem 2.10]{JLR11}.
\end{remark}

\subsection{Tiny and atypical vertices} \label{sec:tiny_atyp}

Even though our main results concern the random graph process rather than the binomial random graph $\Gnp$, in the course of the proof we extensively rely on some of its properties. For small values of $p$, in particular those which give rise to a number of edges as in Theorem \ref{thm:main-thm-pm} and Theorem \ref{thm:main-thm-ham}, a random graph $G_{n,p}$ contains many vertices whose degrees deviate significantly from the average degree. In our proofs those vertices require special attention. The following definition captures them formally.
\begin{definition}
  Given $\delta, p \in [0, 1]$ and a graph $G$ with $n$ vertices, we define the following sets of vertices:
  \begin{align*}
    \TINY_{p, \delta}(G) &= \left\{v \in V(G) \colon \deg_G(v) < \delta np \right\}, \\
    \ATYP_{p, \delta}(G) &= \left\{v \in V(G) \colon \deg_G(v) \notin (1 \pm \delta) np \right\}.
  \end{align*}
  We refer to the vertices in $\TINY_{p, \delta}(G)$ as \emph{tiny} and to the vertices in $\ATYP_{p, \delta}(G)$ as \emph{atypical}.
\end{definition}

Note that in the above definition $G$ is an arbitrary graph, not necessarily a $\Gnp$. This is the reason why $p$ appears in the index as well. In our applications $p$ is usually chosen such that $np$ is roughly (but not necessarily exactly) the average degree of $G$. Thus, as the names indicate, tiny vertices have a degree significantly smaller than the average degree and atypical ones are bounded away from the average by a small constant factor. If $G \sim \Gnp$ then both of these sets become empty as soon as $p$ is large enough: for $p \geq (1 + \eps)\log n/n$ there are a.a.s.\ no tiny vertices for sufficiently small $\delta$ (depending on $\eps$), and similarly for $p \geq C\log n/n$ there are a.a.s.\ no atypical vertices if $C$ is sufficiently large (depending on $\delta$).

With this definition in mind we refine the notion of $\alpha$-resilience.
\begin{definition}\label{def:complicated_resilience}
  Given a graph $G = (V, E)$, a graph property $\cP$, constants $\alpha, \delta_t, \delta_a \in [0, 1]$, and integers $K_t, K_a \in \N$, we say that $G$ is \emph{$(\alpha, \delta_t, K_t, \delta_a, K_a)$-resilient} with respect to $\cP$ if for every spanning subgraph $H \subseteq G$ such that
  \[
    \deg_H(v) \le
      \begin{cases}
        \deg_{G}(v) - K_t, & \text{if } v \in \TINY_{p, \delta_{t}}(G), \\
        \deg_{G}(v) - K_a, & \text{if } v \in \ATYP_{p, \delta_{a}}(G) \setminus \TINY_{p, \delta_{t}}(G), \\
        \alpha\deg_{G}(v), & \text{otherwise},
     \end{cases}
  \]
  for every $v \in V$, where $p = |E| / \binom{|V|}{2}$, we have $G - H \in \cP$.
\end{definition}

Under some natural conditions, one easily sees that $(\alpha,\delta_t,K_t,\delta_a,K_a)$-resilience implies $\alpha$-resilience.
\begin{lemma}\label{lem:resilience-implication}
  Let $G = (V, E)$ be a graph with $n$ vertices and $\alpha, \delta_t, \delta_a \in[0,1]$ and $K_t, K_a \in \N_0$ constants such that $(1 - \alpha)\delta_t np \ge K_a \ge K_t$, where $p = |E| / \binom{n}{2}$. If the minimum degree $d$ of $G$ is such that $d - \floor{\alpha d} \ge K_t$, then $(\alpha, \delta_t, K_t, \delta_a, K_a)$-resilience of $G$ implies $\alpha$-resilience of $G$. \qed
\end{lemma}

In particular, in our applications we have that $K_t$ is either $1$ (for perfect matchings) or $2$ (for Hamilton cycles) and $K_a$ is a (large) constant. A random graph $G_{n, p}$ with $p \gg 1/n$ then a.a.s.\ trivially fulfils the first condition in the lemma above. In addition, for $\alpha = 1/2 - \eps$, the second condition is also satisfied if we constrain $G_{n, p}$ to the subgraph induced by vertices with degree $1$ (perfect matchings) resp.\ the $2$-core (Hamilton cycles).

Next, we establish a couple of properties of tiny and atypical vertices in random graphs.
\begin{lemma}\label{lem:atyp-size}
  Given $\delta > 0$, if $p \geq \log n/(3n)$ then $G \sim \Gnp$ a.a.s.\ satisfies:
  \[
    |\ATYP_{p, \delta}(G)| \leq n/(\log n)^3.
  \]
\end{lemma}
\begin{proof}
  The left hand side consists of all vertices $v\in V(G)$ with degree $\deg(v) \notin (1 \pm \delta)np$. By Chernoff bounds, we have that
  \[
    \Pr[\deg(v) \notin (1 \pm \delta) np] \leq e^{-\Omega_{\delta}(np)} \leq (\log n)^{-4},
  \]
  with room to spare. Thus, Markov's inequality implies that the probability that there are at least $n/(\log n)^3$ vertices with degree not in $(1 \pm \delta)np$ is at most $(\log n)^{-1} = o(1)$, as claimed.
\end{proof}

The crucial observation in our proof strategy is that tiny and atypical vertices cannot be clumped. In particular, no vertex has too many atypical vertices in its proximity.

In the graph process $\{G_i\}$ the property of being tiny or atypical is not monotone. Nevertheless, we expect that these vertices only change slightly over short periods of time. In order to capture this formally, we make use of the fact that one can generate a subsequence of the graph process as follows: choose $p_0$ such that $G^{-} \sim G_{n, p_0}$ has a.a.s.\ just a bit less than $m$ edges, and choose $p'$ such that $G_{n, p'}$ has just a bit more than $\eps m$ edges. Then the union $G^{+} = G^{-} \cup G_{n, p'}$ contains all graphs from $\{G_i\}$ with $m \le i \le (1 + \eps)m$. Moreover, as $p' \ll p_0$ we expect that the tiny and atypical vertices of $G^{-}$ are still scattered, even if we include all the edges from $G^{+}$. The following lemma makes this precise.
\begin{lemma}\label{lem:bad-subgraphs}
  Let $k \geq 2$ be an integer and $\eps > 0$ a constant. There exist positive constants $\delta(\eps)$ and $L(\eps, k)$, such that if $p_0 \geq (1 + \eps)\log n / (kn)$ and $p' \leq \eps p_0$ then a.a.s.\ the following holds. Let $G^{-} \sim G_{n, p_0}$ and set $G^{+} = G^{-} \cup G_{n, p'}$ and $p_1 = 1 - (1 - p_0)(1 - p')$. Then:
  \begin{enumerate}[(i)]
    \item\label{lem-no-tiny-neighbours} for every $v \in V(G^{+})$ we have $\big| N^{3}_{G^+}(v) \cap (\TINY_{p_0, \delta}(G^{-}) \cup \TINY_{p_1, \delta}(G^{+})) \big| \leq k - 1$,
    \item\label{lem-no-atyp-neighbours} for every $v \in V(G^{+})$ we have $\big| N^{3}_{G^+}(v) \cap (\ATYP_{p_0, \delta}(G^{-}) \cup \ATYP_{p_1, \delta}(G^{+})) \big| \leq L$,
    \item\label{lem-no-tiny-cycles} for every cycle $C \subseteq G^{+}$ with $v(C) \leq 2k$ we have
    \[
      \big| V(C) \cap (\TINY_{p_0, \delta}(G^{-}) \cup \TINY_{p_1, \delta}(G^{+})) \big| \le k - 2.
    \]
  \end{enumerate}
\end{lemma}
\begin{proof}
  Throughout the proof we assume that $V(K_n) = V(G^{+})$. We make use of the fact that $G^+$ is distributed as $G_{n, p_1}$.

  $\ref{lem-no-tiny-neighbours}$ We prove a more general statement: if $T \subseteq G^+$ is a tree with $k \le v(T) \le 4k$ vertices, then it contains at most $k - 1$ vertices from $\TINY_{p_0, \delta(G^-)} \cup \TINY_{p_1, \delta}(G^+)$. Suppose that this is the case. Let $v \in V(G^{+})$ be an arbitrary vertex and assume, towards a contradiction, that there are $k$ vertices $u \in \TINY_{p_0, \delta}(G^{-}) \cup \TINY_{p_1, \delta}(G^{+})$ in $N^{3}_{G^{+}}(v)$. Choose $T$ to be a tree which contains $v$ and a shortest path between $v$ and every such vertex $u$. Then $T$ has at least $1 + k$ and at most $1 + 3k$ vertices, and contains $k$ vertices from $\TINY_{p_0, \delta}(G^{-}) \cup \TINY_{p_1, \delta}(G^{+})$, which is a contradiction.

  Let $T \subseteq K_n$ be a tree with $k \leq v(T) \leq 4k$ vertices and consider a subset $S \subseteq V(T)$ of size exactly $k$. Let $\cE_{T}$ denote the event that $T \subseteq G^{+}$ and $\cE_{T, S}$ the event that every vertex in $S$ has at most $\delta np_0$ neighbours in $G^{-}$ which are outside of $V(T)$, or at most $\delta np_1$ neighbours in $G^{+}$ outside of $V(T)$. Note that $\cE_T$ and $\cE_{T, S}$ are independent events. Moreover, if there exists a tree $T \subseteq G^{+}$ which contains at least $k$ vertices from $\TINY_{p_0, \delta}(G^{-}) \cup \TINY_{p_1, \delta}(G^{+})$ then clearly both events $\cE_T$ and $\cE_{T, S}$ happen for $S$ being a subset of size $k$ of such vertices. We show that the probability that $\cE_T \wedge \cE_{T, S}$ happens for any $T, S$ is $o(1)$, which implies the statement.

  First, note that for a fixed tree $T \subseteq K_n$ we have $\Pr[\cE_{T}] = p_1^{v(T) - 1}$. Next, in order for $\cE_{T, S}$ to happen, for each vertex in $S$ we need it to have at most $\delta np_0$ edges incident to it in $G^{-}$, or at most $\delta np_1$ edges in $G^{+}$, with the other endpoint being in $V(G^{+}) \setminus V(T)$. The probability that this happens is at most
  \[
     \Pr[\text{Bin}(n - v(T), p_0) \le \delta np_0] + \Pr[\text{Bin}(n - v(T), p_1) \le \delta np_1].
  \]
  Elementary calculations show that this sum can be bounded by $e^{- (1 - \eps/4) np_0}$, for sufficiently small $\delta$ depending on $\eps$. As the probabilities are independent for different vertices in $S$, we get
  \[
    \Pr[\cE_{T, S}] \leq (e^{- (1 - \eps/4) np_0})^{k}.
  \]
  For each $k \le t \le 4k$ there are at most $\binom{n}{t} t^{t - 2}$ different trees $T \subseteq K_n$ with exactly $t$ vertices, and for each such tree $\binom{t}{k}$ choices for $S \subseteq V(T)$. Therefore, using union bound over all values of $t$ and all such pairs $(T, S)$ we estimate the probability that some $\cE_T \wedge \cE_{T, S}$ happens as follows:
  \[
    \Pr[ \bigcup_{(T, S)} (\cE_{T} \land \cE_{T, S}) ] \leq
    \sum_{t = k}^{4k} \binom{n}{t} \binom{t}{k} t^{t - 2} \cdot \Pr[\cE_{T} \land \cE_{T, S}] \leq
    \sum_{t = k}^{4k} n^{t} t^{t} \cdot p_1^{t - 1} \cdot e^{- k(1 - \eps/4)np_0}.
  \]
  As $p_1 \leq (1 + \eps)p_0$ and $n^t p_0^{t - 1} \cdot e^{- k(1 - \eps/4) np_0}$ is decreasing in $p_0$, this implies
  \[
    \Pr[ \bigcup_{(T, S)} (\cE_{T} \land \cE_{T, S}) ] = O_{\eps,k}( n\cdot (\log n )^{4k} \cdot n^{- 1 - \eps/2} ) = o(1).
  \]

  $\ref{lem-no-atyp-neighbours}$ Similarly to the previous case, we prove a statement for trees that implies the desired result. Let $T \subseteq K_n$ be a tree with $L \leq v(T) \leq 4L$ and let $S \subseteq V(T)$ be a set of $L$ vertices. Let $\cE_{T}$ denote the event that $T \subseteq G^{+}$ and $\cE_{T, S}$ the event that every vertex $v \in S$ satisfies either $|N_{G^{-}}(v) \setminus V(T)| \notin (1 \pm \delta/2)np_0$ or $|N_{G^{+}}(v) \setminus V(T)| \notin (1 \pm \delta/2)np_1$. If there exists $T \subseteq G^{+}$ with at least $L$ vertices belonging to $\ATYP_{p_0, \delta}(G^{-}) \cup \ATYP_{p_1, \delta}(G^{+})$, both $\cE_{T}$ and $\cE_{T, S}$ happen (where once again $S$ is any subset of $L$ such vertices). For a fixed vertex $v \in S$, Chernoff bounds show that the probability that $v$ satisfies $|N_{G^{-}}(v) \setminus V(T)| \notin (1 \pm \delta/2)np_0$ or $|N_{G^{+}}(v) \setminus V(T)| \notin (1 \pm \delta/2)np_1$ is at most
  \[
    \Pr[\cE_{T, S}] \leq \Pr[\text{Bin}(n - v(T), p_0) \notin (1 \pm \delta/2)np_0] + \Pr[\text{Bin}(n - v(T), p_1) \notin (1 \pm \delta/2)np_1] \leq e^{-\gamma np_0},
  \]
  for some $\gamma = \gamma(\eps) > 0$ (for $\delta=\delta(\eps)$ sufficiently small). Note that these events are independent for different vertices in $S$ and therefore
  \[
    \Pr[\cE_{T, S}] \leq (e^{-\gamma np_0})^{L}.
  \]
  Analogous union bound analysis as in part $\ref{lem-no-tiny-neighbours}$ yields
  \[
    \Pr[ \bigcup_{(T, S)} (\cE_{T} \land \cE_{T, S}) ] = O_{\eps, k}( n \cdot (\log n )^{4L} \cdot e^{- L\gamma np_0} ) = o(1),
  \]
  for $L$ large enough depending on $\eps$ and $k$.

  $\ref{lem-no-tiny-cycles}$ Let $C \subseteq K_n$ be a cycle with $3 \leq v(C) \leq 2k$ vertices and consider a subset $S \subseteq V(C)$ of size exactly $k - 1$. Let $\cE_{C}$ and $\cE_{C, S}$ denote the events as in $\ref{lem-no-tiny-neighbours}$ with $T$ replaced by $C$. Assuming that there exists a cycle $C \subseteq G^{+}$ with at least $k - 1$ vertices from $\TINY_{p_0, \delta}(G^{-}) \cup \TINY_{p_1, \delta}(G^{+})$, both events $\cE_{C}$ and $\cE_{C, S}$ happen (again, for $S$ being any subset of $k - 1$ such vertices). Analogous analysis as before shows
  \[
    \Pr[\cE_{C, S}] \leq (e^{- (1 - \eps/4) np_0})^{k - 1},
  \]
  On the other hand, we have $\Pr[\cE_{C}] = p_1^{v(C)}$. For each $3 \leq \ell \leq 2k$ there are less than $n^\ell$ cycles $C \subseteq K_n$ with exactly $\ell$ vertices, and for each such cycle at most $\binom{\ell}{k - 1}$ many choices for $S \subseteq V(C)$. A union bound over all values of $\ell$ and all possible pairs $(C, S)$, shows that the probability that the property $\ref{lem-no-tiny-cycles}$ fails is at most
  \[
    \Pr [ \bigcup_{(C, S)} (\cE_{C} \land \cE_{C, S})] \leq \sum_{\ell = 3}^{2k} n^{\ell} \binom{\ell}{k - 1} \cdot \Pr[\cE_{C} \land \cE_{C, S}] \leq \sum_{\ell = 3}^{2k} n^{\ell} \ell^{k} \cdot p_1^{\ell} \cdot e^{- (k - 1)(1 - \eps/4) np_0}.
  \]
  We once again obtain
  \[
    \Pr [ \bigcup_{(C, S)} (\cE_{C} \land \cE_{C, S})] = O_{\eps, k}( (\log n)^{2k} \cdot n^{-(k - 1)(1 + \eps/2)/k} ) = o(1),
  \]
  and the property follows.
\end{proof}

We note that the proof actually shows that we could replace $N_{G^+}^3(v)$ with $N_{G^+}^\ell(v)$ for any constant $\ell \in \N$. For our purposes, the third neighbourhood suffices.

\subsection{Properties of random graphs}\label{sec:random_graphs}
In this subsection we establish some further, more general properties of random graphs which we rely on in the later sections. The following is a well known bound on the number of edges between sets of vertices in random graphs (see, e.g.~\cite[Corollary 2.3]{krivelevich2006pseudo}).
\begin{lemma}\label{lem:gnp-edge-distribution}
  Let $p = p(n) \leq 0.99$. Then $G \sim \Gnp$ a.a.s.\ has the property that for every two (not necessarily disjoint) subsets $X, Y \subseteq V(G)$ the number of edges with one endpoint in $X$ and the other in $Y$ satisfies:
  \[
    |e_G(X, Y) - |X||Y|p| \le c \sqrt{|X||Y|np},
  \]
  for some absolute constant $c > 0$.
\end{lemma}

Lastly, we show that if $p$ is not too large then the neighbourhood of every two vertices is almost disjoint.
\begin{lemma}\label{lem:small-codegree}
  If $p = o(n^{-5/6})$ then $G \sim \Gnp$ a.a.s.\ has the property that every two distinct vertices have at most two common neighbours, i.e.\ $|N_G(u) \cap N_G(v)| \leq 2$ for all distinct $u, v \in V(G)$.
\end{lemma}
\begin{proof}
  It is sufficient to show that there are no two vertices with a common neighbourhood of size three. The probability that there exists a pair of vertices $u, v \in V(G)$ violating this property is at most
  \[
    \binom{n}{2} \binom{n - 2}{3} p^{6} \leq n^{5} p^{6} = o(1).
  \]
  This completes the proof.
\end{proof}

%!TEX root = threshold_resilience.tex
\section{Perfect matchings}\label{sec:matching}

Let $\{G_i\}$ denote the random graph process. Recall that Theorem~\ref{thm:main-thm-pm} states that a.a.s.\ for every $m \ge \frac{1 + \eps}{4} n \log n$ we have that the subgraph of $G_m$ obtained by removing isolated vertices is resilient with respect to containing a perfect matching. As remarked earlier, $G_m$ has the same distribution as $\Gnm$ thus an obvious way to prove this statement would be to estimate for each $m$ the probability that $\Gnm$ has the resilience property and apply a union bound. Unfortunately, this probability is roughly $1 - e^{- \alpha \cdot 2m/n}$, for some small constant $\alpha > 0$, which is too weak to cover all the values of $m$ which are of order $n\log n$. This is not surprising as typically random graphs fail to satisfy a property involving spanning subgraphs with a probability that is only exponential in the average degree. We go around this issue by using the following strategy: rather than handling each graph individually, we bundle consecutive graphs in groups of size roughly $\eps n \log n$. The following proposition shows that almost surely \emph{all} graphs in such a group have the desired property. A union bound over constantly many groups thus implies that the theorem holds for all $m \in \{\frac{1 + \eps}{4} n \log n, \ldots, C n \log n\}$, for some large constant $C$ of our choice. The remaining values of $m$ can then be treated one at a time using a result of Sudakov and Vu~\cite{SV08}, which gives a probability of at least $1 - 1/n^3$ for all $m \ge C n\log n$.
\begin{proposition}\label{prop:fixed-range-pm}
  For every constant $\eps > 0$ and integer $\frac{1 + \eps}{4} n \log n \le m_0 \le n (\log n)^2$, there exist positive constants $\delta_{t}(\eps)$, $\delta_{a}(\eps)$, and $K(\eps)$ such that the random graph process $\{G_i\}$ a.a.s.\ has the following property: For any integer $m_0 \leq m \leq (1 + \eps/4)m_0$ the graph obtained by removing all isolated vertices from $G_m$ is $(1/2 - \eps, \delta_t, 1, \delta_a, K)$-resilient with respect to having a perfect matching.
\end{proposition}

We remark that the fact that we use the same $\eps$ in the lower bound for $m_0$ as we do for the resilience is without loss of generality, as we can always choose the smaller of the two for both.

\paragraph{Proof overview.} The proof itself is somewhat technical, so let us first give a brief overview before we dive into the details. To handle all graphs $G_m$ for $m \in [m_0, (1 + \eps/4)m_0]$ simultaneously we change the way $G_m$ is generated (akin to what has been indicated prior to Lemma \ref{lem:bad-subgraphs}). Instead of generating $G_m$ starting from an empty graph, we obtain it as follows: sample $G^{-} \sim G_{n, p_0}$ and $G_{n, p'}$, and choose a random ordering $\pi$ of the edges in $G_{n, p'}$. In particular, we choose $p_0$ and $p'$ such that almost surely $G^{-}$ has less than $m_0$ edges (but not too much less) and $G^{+} = G^{-} \cup G_{n, p'}$ has more than $(1 + \eps/4)m_0$ edges (again, not too much more). Now we can generate each $G_m$ as a union of $G^{-}$ and the first $m - e(G^{-})$ edges according to $\pi$; it is a simple exercise to show that this is the same as generating $G_m$ from `scratch'.

It is crucial that all the properties of $G^{-}$ and $G^{+}$ that we use (e.g.\ upper bound on the number of edges between certain sets in $G^{+}$, tiny and atypical vertices being far apart in both $G^{-}$ and $G^{+}$, etc.), are such that they are also satisfied by all graphs `squeezed' in between them. While this is a standard technique in showing hitting-time results, it raises problems in the resilience setting. For example, vertices can be tiny or atypical in $G_m$ without being tiny or atypical in $G^{-}$ or $G^{+}$.

We circumvent this by defining {\em tiny} vertices to be those vertices that are tiny in at least one of $G^{-}$ or $G^{+}$ with respect to a parameter $\delta_t'$ which is somewhat larger than the value of $\delta_t$ in the definition of $(1/2 - \eps, \delta_t, 1, \delta_a, K)$-resilience in Proposition~\ref{prop:fixed-range-pm}. In this way we can guarantee that every vertex in $G$ that is tiny with respect to $\delta_t$ is also tiny in $G^{-}$ or $G^{+}$ with respect to $\delta_t'$. If we thus allow an adversary to remove all but one incident edge for all vertices that are tiny in $G^{-}$ or $G^{+}$ with respect to $\delta_t'$, this may result in the removal of more edges than allowed by the definition of $(1/2 - \eps, \delta_t, 1, \delta_a, K)$-resilience. Nevertheless, Lemma~\ref{lem:bad-subgraphs} still guarantees that all those vertices are sufficiently far apart in $G^{+}$ and thus also in $G \subseteq G^{+}$. This, in turn, implies that all these vertices can be covered in the matching even though they all have only one edge left. Atypical vertices are handled similarly by considering atypical vertices in $G^{-}$ or $G^{+}$ with respect to a parameter $\delta_a'$ that is somewhat smaller than the value of $\delta_a$ in the definition of $(1/2 - \eps, \delta_t, 1, \delta_a, K)$-resilience. The rest of the proof then follows standard arguments for finding a perfect matching in sparse random graphs (see, e.g.~\cite{JLR11,luczak1991tree}): greedily match all tiny and atypical vertices, in that order, and show that the remaining graph can be equipartitioned in a way that the resulting bipartite graph satisfies Hall's matching criteria.
\begin{proof}
  Let $p_0 = (1 - \eps/16) m_0 / \binom{n}{2}$, $p' = (\eps/2) p_0$, and let $G^{+}$ be the union of independent random graphs $G^{-} \sim G_{n, p_0}$ and $G_{n, p'}$. Then $G^{+}$ is distributed as $G_{n,p_1}$, where $p_1 = 1 - (1 - p_0)(1 - p')$. By Lemma~\ref{lem:gnp-edge-distribution} we have that the number of edges in $G^{-}$ is a.a.s.\ at most $m_0$ and the number of edges in $G^{+}$ is a.a.s.\ at least $(1 + \eps/4) m_0$.

  Let $\delta_{t}' = \delta_{{\ref{lem:bad-subgraphs}}}(\eps/2)$, $\delta_{a}' = \min\{ \eps/64, \delta_{{\ref{lem:bad-subgraphs}}}(\eps/2) \}$, $c = c_{\ref{lem:gnp-edge-distribution}}$, $L = \max\{L_{\ref{lem:bad-subgraphs}}(\eps/2), L_{\ref{cl:no-small-pm-deg-tree}}(\eps)\}$. With this at hand, we set the constants given in the statement of the proposition as follows: $\delta_{t} = \delta_{t}'/40$, $\delta_{a} = \max\{ \eps/4, 16\delta_{a}' \}$, and $K = 2L$.

  For the rest of the proof consider some $m_0 \le m \le (1 + \eps/4)m_0$, and let $G \subseteq G_m$ be the subgraph obtained by removing all isolated vertices from $G_m$. Let $V = V(G)$ denote its vertex set and let $\TINY$ and $\ATYP$ be sets of vertices defined as:
  \begin{align*}
    \TINY &:= (\TINY_{p_0, \delta_{t}'}(G^{-}) \cup \TINY_{p_1, \delta_{t}'}(G^{+})) \cap V, \\
    \ATYP &:= (\ATYP_{p_0, \delta_{a}'}(G^{-}) \cup \ATYP_{p_1, \delta_{a}'}(G^{+})) \cap V.
  \end{align*}
  By Lemma \ref{lem:atyp-size} there are at most $n/\log^3 n$ isolated vertices in $G^-$, thus there are at most that many in $G_m$ as well. Our choice of constants then implies that for all $m_0 \leq m \leq (1 + \eps/4)m_0$ we have
  \begin{align*}
    \TINY_{p, \delta_{t}}(G) \subseteq \TINY
    \qquad\mbox{and}\qquad
    \ATYP_{p, \delta_{a}}(G) \subseteq \ATYP,
  \end{align*}
  where $p = m/\binom{v(G)}{2}$ denotes the density of $G$ (note that $E(G) = E(G_m)$ as we only remove isolated vertices). As noted in the discussion before the proof, if we show that for every graph $H$ of the form
  \begin{equation}\label{eq:weaker}
    \deg_H(v) \leq
      \begin{cases}
        \deg_{G}(v) - 1, & \text{if } v \in \TINY, \\
        \deg_{G}(v) - K, & \text{if } v \in \ATYP \setminus \TINY, \\
        (1/2 - \eps)np_1, & \text{otherwise},
      \end{cases}
  \end{equation}
  the graph $G - H$ contains a perfect matching, then this implies that $G$ is $(1/2 - \eps, \delta_t, 1, \delta_a, K)$-resilient with respect to having a perfect matching. Note that this follows from the fact that typical vertices in $G$ have degree at most $(1 + \delta_{a})np$ and thus removing an $(1/2 - \eps)$-fraction of their degree is less than removing $(1/2 - \eps)np_1$ incident edges, since $p_1 > (1 + \eps/4)p$. In the remainder of the proof we show that the graph $G - H$ indeed contains a perfect matching.

  First, note that $G$ satisfies a series of properties:
  \begin{enumerate}[leftmargin=3.5em, font={\bfseries\itshape}, label=(M\arabic*)]
    \item\label{M-max-degree} the maximum degree of $G$ is at most $\Delta(G) \leq \log^3 n$,
    \item\label{M-num-edges} for all $X, Y \subseteq V$ we have $e_{G}(X, Y) \leq |X||Y|p_1 + c\sqrt{|X||Y|np_1}$,
    \item\label{M-codegree} for all $v, u \in V$ we have $|N_{G}(v) \cap N_{G}(u)| \leq 2$,
    \item\label{M-second-neighbourhoods} for all $v \in V$ we have $|N^{2}_{G}(v) \cap \TINY| \leq 1$ and $|N^{2}_{G}(v) \cap \ATYP| \leq L$.
    \item\label{M-size-atypical} $|\ATYP| \leq \frac{2n}{\log^3 n}$.
  \end{enumerate}
  We show that the properties hold in $G^{+}$, and hence in any subgraph $G \subseteq G^{+}$. Indeed, we have $p \leq 10 \log^2 n/n$ (by our assumptions on $m_0$) thus \ref{M-max-degree} follows from simple bounds on the binomially distributed random variable and a union bound over all vertices. Property \ref{M-num-edges} is given by Lemma~\ref{lem:gnp-edge-distribution} and \ref{M-codegree} by Lemma~\ref{lem:small-codegree}. Next, \ref{M-second-neighbourhoods} holds by our choice of $L$ and by Lemma~\ref{lem:bad-subgraphs} applied with $k = 2$ and $\eps/2$ as $\eps$. Note that we can indeed apply Lemma~\ref{lem:bad-subgraphs} with these parameters as $p_0 \geq (1 + \eps/2) \log n/(2n)$. Lastly, \ref{M-size-atypical} holds by Lemma~\ref{lem:atyp-size}.

  Consider graph $H$ which satisfies \eqref{eq:weaker}, and let $G' = G - H$ and $U := V \setminus \ATYP$. If $G$ has an odd number of vertices then we remove a vertex $v \in \TINY$ which is incident to $u \notin \TINY$, and if such a vertex does not exist, then we remove one from $V \setminus \TINY$. Owing to the property \ref{M-second-neighbourhoods} this does not decrease the degree of any vertex from $\TINY$ and decreases the degree of a vertex from $V \setminus \TINY$ by at most one. Note that all vertices $v \in U$ have degree in $G^{-}$ at least $(1 - \delta_{a}')np_0$, hence by \ref{M-second-neighbourhoods} we have:
  \begin{align*}
    \deg_{G'}(v, U) &\geq \deg_{G^{-}}(v) - \deg_{H}(v) - \deg_{G}(v, \ATYP) - 1 \\
    &\geq (1 - \delta_{a}')np_0 - (1/2 - \eps)np_1 - L - 1 \\
    &\geq (1/2 - \delta_{a}' + \eps/2)np_1 - L - 1 \\
    &\geq (1/2 + \eps/4)np_1,
  \end{align*}
  as $\delta_{a}' \leq \eps/64$ and $p_0 \geq p_1/(1 + \eps/2)$.

  Due to property \ref{M-second-neighbourhoods} it is straightforward to greedily find a matching saturating all vertices from $\TINY$. Next, we match all remaining vertices from $\ATYP$. Again we proceed greedily. Consider an arbitrary, still unmatched vertex $w \in \ATYP \setminus \TINY$. Note that $\deg_{G'}(w) \geq K - 1$ (here $-1$ comes from the fact that we might have removed one vertex from $G$ to achieve an even number of vertices). By \ref{M-second-neighbourhoods} and our choice of $K$ we have $|N^{2}_{G'}(w) \cap \ATYP| \leq L \leq K/2$. Therefore, at most $K/2$ neighbours of $w$ have been matched so far, thus there is an available one.

  Let $V(M)$ be the set of all vertices saturated in this partial matching. In particular $V(M)$ contains all vertices of $\ATYP$. Set $U_1 := U \setminus V(M)$ and let $G'' := G'[U_1]$ be the subgraph induced by the remaining vertices. Observe that by property \ref{M-second-neighbourhoods} the degree of any vertex in $U_1$ decreases by at most $2L$ with respect to its degree in $G'$, hence for all vertices $v \in U_1$ we have
  \[
    \deg_{G''}(v) \geq \deg_{G'}(v, U) - \deg_{G'}(v, V(M)) \geq (1/2 + \eps/4)np_1 - 2L \geq (1/2 + \eps/5)np_1.
  \]
  Take $U_1 = A \cup B$ to be a uniformly at random chosen balanced bipartition of the set $U_1$. We now define the set of all vertices which have significantly less than the expected degree in either $A$ or $B$ as
  \[
    D := \{ v \in U_1 \colon \deg_{G''}(v, A) < (1/2 + \eps/7)|A|p_1 \text{ or } \deg_{G''}(v, B) < (1/2 + \eps/7)|B|p_1 \}
  \]
  and call all such vertices {\em degenerate}.

  Firstly, we show that there are not many degenerate vertices. As $|A|, |B| \geq (1/2 - o(1))n$ by property \ref{M-size-atypical}, we have from Chernoff bounds for hypergeometrically distributed random variables that a fixed vertex $v$ is degenerate with probability
  \[
    \Pr[v \text{ is degenerate}] \leq 2 \Pr \left[\deg_{G''}(v, A) < (1/2 + \eps/7)|A|p_1 \right] \leq e^{-2\gamma np_1} \leq n^{-2\gamma},
  \]
  for some $\gamma=\gamma(\eps) > 0$. Consequently, by Markov's inequality there are at most $n^{1 - \gamma}$ degenerate vertices, i.e.\ $|D| \leq n^{1 - \gamma}$. Next we show that the degenerate vertices cannot be `too close' in $G''$.
  \begin{claim}\label{cl:no-small-pm-deg-tree}
    There exists a positive constant $L(\eps)$ such that a.a.s.\ for every $v \in U_1$ we have $|N^{2}_{G''}(v) \cap D| < L$.
  \end{claim}
  \begin{proof}
    Consider some vertex $v \in V(G'')$ and a subset $D_v \subseteq N_{G''}^2(v)$ of size $L$. What is the probability that a random equipartition $A \cup B$ of $V(G'')$ makes all the vertices in $D_v$ degenerate?

    First, by property \ref{M-codegree} we know that for every vertex $u \in D_v$ there are at least
    \begin{equation}\label{eq:lower_bound_Dv}
      (1/2 + \eps/5)np_1 - 2L \geq (1/2 + \eps/6)np_1
    \end{equation}
    vertices in its neighbourhood in $G''$ which do not belong to the neighbourhood of any other vertex in $D_v \setminus \{u\}$. Let us denote such vertices with $N^*_u$.

    If all vertices in $D_v$ are degenerate, then there has to exist a subset $D_v' \subseteq D_v$ of size at least $L/2$ such that all vertices in $D_v'$ either have too few neighbours into $A$ or too few neighbours into $B$. By symmetry we may assume that this set is $A$. That is, we have $|A \cap N_u^*| \le (1/2 + \eps/7)|A|p_1$ for every $u \in D_v'$. Therefore,
    \[
      \Big| A \cap \bigcup_{u \in D_v'} N_u^* \Big| \le (1/2 + \eps/7)|A|p_1 |D_v'|.
    \]
    Let $\cE_{D_v'}$ denote the event of this happening. From \eqref{eq:lower_bound_Dv} we have
    \[
      \Big| \bigcup_{u \in D_v'} N_u^* \Big| \ge (1/2 + \eps/6)np_1 |D_v'|.
    \]
    Thus, as $A$ is a randomly chosen subset of linear size, Chernoff bounds for hypergeometrically distributed random variables show
    \begin{equation}\label{eq:one_vertex_probability}
      \Pr[\cE_{D_v'}] \leq e^{- \gamma L np_1},
    \end{equation}
    for some $\gamma = \gamma(\eps) > 0$ not depending on $L$. To summarise, we can bound the probability that all the vertices in $D_v$ are degenerate by applying \eqref{eq:one_vertex_probability} together with a union bound over all possible choices for $D_v'$. Finally, we take a union bound over all vertices $v$ and sets $D_v$. There are $n$ choices for $v$ and at most $\binom{(\log^3 n)^2}{L}$ choices for $D_v \subseteq N_{G''}^2(v)$ of size $L$ (which follows from \ref{M-max-degree}). The expected number of vertices for which there exists a set $D_v$ of size $L$ consisting solely of degenerate vertices is then at most
    \[
      n \cdot (\log n)^{6L} \cdot 2^{L} \cdot e^{- \gamma L np_1} = o(1),
    \]
    for $L > 0$ sufficiently large. In other words, for $L$ large enough no such set exists with probability $1 - o(1)$, as claimed.
  \end{proof}

  Similarly as before, we greedily construct a partial matching $M_D$ that saturates all the degenerate vertices. For an arbitrary degenerate vertex $v \in D$ we have $\deg_{G''}(v) \geq (1/2 + \eps/5)np_1$ and as there cannot be more than $L$ degenerate vertices in $N_{G''}^{2}(v)$ by Claim~\ref{cl:no-small-pm-deg-tree}, there is a vertex available to match $v$ to.

  Let $V(M_D)$ be the set of all vertices saturated in this partial matching and let $A' := A \setminus V(M_D)$ and $B' := B \setminus V(M_D)$. Again by Claim~\ref{cl:no-small-pm-deg-tree}, for all $v \in (A \cup B) \setminus V(M_D)$ we get
  \[
    \deg_{G''}(v, A') \geq (1/2 + \eps/7)|A|p_1 - 2L \geq (1/4 + \eps/16)np_1,
  \]
  and analogously
  \[
    \deg_{G''}(v, B') \geq (1/2 + \eps/7)|B|p_1 - 2L \geq (1/4 + \eps/16)np_1,
  \]
  as $|A|, |B| > (1/2 - o(1))n$. However, we might not have $|A'| = |B'|$ any more. Assume w.l.o.g.\ that $|A'| > |B'|$ and note that $|A'| \leq |B'| + |V(M_D)|$. In order to find a balanced bipartition we thus need to redistribute at most $|D| \leq n^{1 - \gamma}$ vertices, for some $\gamma > 0$. To achieve this we build a $2$-\emph{independent set} in $A'$, i.e.\ an independent set in which no two vertices have a common neighbour, of size at least $n^{1 - \gamma}$. Recall, from property \ref{M-max-degree} we have $\Delta(G) \leq \log^{3} n$, thus a straightforward greedy construction shows that there exists a $2$-independent set of size at least $n/\log^{7} n \geq n^{1 - \gamma}$ in $A'$, which is more than enough for our purposes.

  Let $A''$ and $B''$ be the two sets after moving the vertices belonging to the $2$-independent set into $B'$. Then $|A''| = |B''|$ (by construction) and vertices in $A''$ (respectively, $B''$) have degree at least $(1/4 + \eps/16)np_1 - 1 \ge (1/4 + \eps/17)np_1$ in $B''$ (respectively, $A''$), as we moved a $2$-independent set.

  It remains to find a perfect matching in the bipartite graph $G''[A'', B'']$. This is easily achieved through Hall's matching criteria (see, e.g.~\cite{diestel2000graph}). Recall its statement: if each subset $S \subseteq A''$ of size at most $|A''|/2$ has at least $|S|$ neighbours in $B''$ (and similarly for $S \subseteq B''$), then $G''[A'', B'']$ contains a perfect matching. We now verify that this is indeed the case.

  Let $S \subseteq A''$ be an arbitrary subset of size $s$. The number of edges between any two disjoint subsets of $G''$ of size $s \leq n/4$ is at most $s(np_1/4 + c\sqrt{np_1})$, by property \ref{M-num-edges}. On the other hand however, the minimum degree condition in $G''$ yields that the number of edges between $S$ and $B''$ is at least $s (1/4 + \eps/17)np_1$, which implies $|N(S)| > s$. This completes the proof of the proposition.
\end{proof}

Having Proposition \ref{prop:fixed-range-pm} at hand we now complete the proof. Let us first, as promised in the introduction, restate Theorem~\ref{thm:main-thm-pm} in terms of stronger resilience.
\begin{theorem}\label{thm:stronger-resilience-pm}
  Let $\eps > 0$ be a constant and consider the random graph process $\{G_i\}$. There exist positive constants $\delta_{t}(\eps)$, $\delta_{a}(\eps)$, and $K(\eps)$ such that a.a.s.\ for every $m \geq \frac{1 + \eps}{4} n\log n$ we have that the graph obtained from $G_m$ by deleting all isolated vertices is $(1/2 - \eps, \delta_t, 1, \delta_a, K)$-resilient with respect to containing a perfect matching.
\end{theorem}
\begin{proof}
  Let $C$ be a sufficiently large constant. If $m \ge C n \log n$ then the random graph $G_{n,m}$ does not contain atypical vertices. Therefore, by \cite[Theorem 3.1]{SV08} the statement holds for every such (fixed) $m$ with probability at least $1 - e^{-\alpha \cdot C \log n}$, for some small constant $\alpha > 0$, thus a union bound implies that it holds for all $m \ge C n \log n$ simultaneously. As for the rest, consider intervals of the form
  \[
    \bigg[ \frac{1 + i\eps}{4} n \log n, \frac{1 + (i + 1)\eps}{4} n \log n \bigg)
  \]
  for $i \in \{1, \ldots, C_{\eps} \}$, where $C_{\eps}$ is such that the last interval contains $C n \log n$. For each interval the conclusion of Proposition~\ref{prop:fixed-range-pm} holds with probability $1 - o(1)$, thus a union bound shows that it a.a.s.\ holds for all intervals. This concludes the proof.
\end{proof}

\begin{proof}[Proof of Theorem~\ref{thm:main-thm-pm}]
  The assertion follows directly from Lemma~\ref{lem:resilience-implication} and Theorem~\ref{thm:stronger-resilience-pm}.
\end{proof}

%!TEX root = threshold_resilience.tex
\section{Hamiltonicity} \label{sec:HAM}

The so-called P\'{o}sa's \emph{rotation-extension} method introduced in \cite{posa1976hamiltonian} is nowadays a standard approach for constructing Hamilton cycles in random graphs, cf.\ also~\cite{ben2011resilience, FK15, frieze2008two, krivelevich2015long, KLS14}. The problem with having edge probabilities near (or even below) the connectivity threshold, is that then the random graph does not satisfy the expansion properties needed to apply the method. We go around this by partitioning the vertex set into typical and atypical vertices. The subgraph induced by the vertices of typical degree satisfies the expansion properties required to apply the rotation-extension technique and we thus get a Hamilton cycle in this subgraph by a standard approach. Our new contribution is to show how to extend this cycle to also contain all atypical vertices.

Towards this goal we make use of the facts given by Lemma~\ref{lem:bad-subgraphs}, that is that atypical vertices do not clump. This allows us to modify the rotation-extension procedure from~\cite{LS12} such that it only uses the expansion properties of typical vertices. In the next section we review some basic notions and state necessary lemmas which are then used in the subsequent section to derive Theorem \ref{thm:main-thm-ham}.

\subsection{Backbone graphs and boosters}

The central notion of the rotation-extension method is that of {\em boosters}. A booster is a non-edge in a graph $G$ whose existence would increase the length of a longest path in $G$ or close a Hamilton path to a Hamilton cycle. The idea behind the rotation-expansion technique is that a graph which is not Hamiltonian contains so many boosters that adding a few random edges is highly likely to contain one of them. The name `rotation-extension' comes from the way how boosters are obtained (\emph{rotation}) and the fact that a longest path of every non-Hamiltonian graph can be increased by adding an edge in place of a booster (\emph{extension}). We now make this precise.
\begin{definition}[Boosters]
  Given a graph $\Gamma$, we say that a non-edge $\{u, v\} \notin E(\Gamma)$ is a \emph{booster} with respect to $\Gamma$, if either $\Gamma + \{u, v\}$ is Hamiltonian or adding $\{u, v\}$ to $\Gamma$ increases the length of a longest path. For a vertex $v \in V(\Gamma)$, we denote by $B_\Gamma(v)$ the set of boosters associated with $v$:
  \[
    B_\Gamma(v) = \{ u \in V(\Gamma) \setminus (N_\Gamma(v) \cup \{v\}) \colon \{v, u\} \text{ is a booster}\}.
  \]
\end{definition}

A standard strategy for implementing the rotation-extension technique is to split the given graph into two graphs: a `backbone' graph responsible for obtaining boosters, and the remainder responsible for finding real edges corresponding to boosters. As we are dealing with subgraphs of random graphs, rather than random graphs themselves, it is convenient to capture the main pseudorandom properties which are used.
\begin{definition}[Backbone graph]\label{def:backbone}
  Given $\alpha, q \in (0, 1)$ and an integer $K \geq 1$, we say that a graph $\Gamma$ with $n$ vertices is an $(\alpha, K, q)$-{\em backbone graph} if there exists a partition of its vertex set $V(\Gamma) = U \cup W_1 \cup W_2$ such that the following holds:
  \begin{enumerate}[leftmargin=2.8em, font={\bfseries\itshape}, label=(P\arabic*)]
    \item\label{P-size-bad} $|W_1 \cup W_2| \leq \frac{n}{\log^2 n}$,
    \item\label{P-degree-bad} for every $v \in W_1$ we have $\deg_{\Gamma}(v) = 2$ and for every $u \in W_2$ we have $\deg_{\Gamma}(u) \geq 2K$,
    \item\label{P-second-W1} for every $v \in W_1$ we have $N_\Gamma^2(v) \cap W_1 = \varnothing$,
    \item\label{P-second-neighbourhood} for every $v \in V(\Gamma)$ we have $|N_{\Gamma}^2(v) \cap W_1| \leq 2$ and $|N_{\Gamma}^2(v) \cap W_2| \leq K$, and
    \item\label{P-expansion} for all $S \subseteq U$ we have
    \[
      |N_{\Gamma}(S)| \geq
        \begin{cases}
          |S| \sqrt{nq}, & \text{if } |S| < K/q, \\
          (1/2 + \alpha/2)n, & \text{if } |S| \geq K/q.
        \end{cases}
    \]
  \end{enumerate}
\end{definition}

The role of the sets $W_1$ and $W_2$ and properties \ref{P-size-bad}--\ref{P-second-neighbourhood} is to capture properties of tiny and atypical vertices in a random graph with density $q$. Property \ref{P-expansion} states that the subgraph induced by typical vertices has good expansion properties.

The next lemma shows that a backbone graph contains many boosters. It can be proven similarly as \cite[Lemma 3.2]{LS12}, with slight modifications which allow us to deal with the vertices in $W_1$ and $W_2$. We defer the proof to Section~\ref{sec:rotation-extension}.
\begin{lemma}\label{lem:boosters}
  For every $\alpha, \delta > 0$, there exists a positive constant $K(\alpha)$ such that the following holds for $q \geq \delta \log n/n$. Let $G_q$ be a graph with $n$ vertices satisfying the property of Lemma~\ref{lem:gnp-edge-distribution} for $q$ (as $p$) and some constant $c'$ (as $c$), and $H_q \subseteq G_q$ a graph with $\Delta(H_q[U]) \leq (1/2 - 2\alpha)nq$, for some $U \subseteq V(G_q)$. Then an $(\alpha, K, q)$-backbone graph $\Gamma$ with a witness partition $V(\Gamma) = U \cup W_1 \cup W_2$, such that $\Gamma[U] = G_q[U] - H_q[U]$, is either Hamiltonian or there are at least $(1/2 + \alpha)n$ vertices $v \in U$ such that $|B_{\Gamma}(v) \cap U| \ge (1/2 + \alpha)n$.
\end{lemma}

Note that rather than asking for $\Gamma = G_q - H_q$, we only require the subgraph of $\Gamma$ induced by $U$ to be given by $G_q[U] - H_q[U]$. The way we exploit such a weaker requirement becomes apparent in the next section. As the notation $q$ in the previous lemma already indicates, we obtain an $(\alpha, K, q)$-backbone graph $\Gamma$ by considering a random subgraph with density $q$ of some graph $G$.

The next lemma states that given a sufficiently sparse graph $\Gamma$ with many boosters, a random graph with appropriate density is likely to contain plenty of edges corresponding to them.
\begin{lemma}\label{lem:extension}
  For every $\alpha > 0$ there exists a positive constant $\mu(\alpha)$ such that for $p \geq \log n/(3n)$, the random graph $G \sim G_{n, p}$ a.a.s.\ satisfies the following. Let $\Gamma$ be a graph with $e(\Gamma) \leq \mu n^2 p$ and $U \subseteq V(G)$ a subset of vertices such that $\Gamma[U] \subseteq G$. If there are at least $(1/2 + \alpha)n$ vertices $v \in U$ such that $|B_{\Gamma}(v) \cap U| \geq (1/2 + \alpha)n$, then there exists a vertex $v \in U$ satisfying $|N_{G}(v, B_\Gamma(v) \cap U)| > np/2$.
\end{lemma}
The proof of the lemma is fairly standard and goes along the lines of the proof of \cite[Lemma 3.5]{LS12}; we include it for completeness in Section~\ref{sec:rotation-extension}.

\subsection{Proof of Theorem~\ref{thm:main-thm-ham}} \label{sec:main_ham}

Similarly as in the case of perfect matchings, instead of proving Theorem~\ref{thm:main-thm-ham} directly, we first show a proposition which considers only a small range of $m$, until $m = C n\log n$.
\begin{proposition}\label{prop:fixed-range-ham}
  For every constant $\eps > 0$ and integer $m_0 \geq (1 + \eps) \frac{n \log n}{6}$ there exist positive constants $\delta_{t}(\eps)$, $\delta_{a}(\eps)$, and $K(\eps)$ such that the random graph process $\{G_i\}$ a.a.s.\ has the following property: For every integer $m_0 \leq m \leq (1 + \eps/6)m_0$, the $2$-core of $G_m$ is $(1/2 - \eps, \delta_{t}, 2, \delta_{a}, K)$-resilient with respect to having a Hamilton cycle.
\end{proposition}

The general strategy of the proof is similar as in the case of perfect matchings. We first sample two random graphs $G^{-} \sim G_{n, p_0}$ and $G_{n, p'}$ and look at their union $G^{+} = G^{-} \cup G_{n, p'}$. The choice of densities $p_0$ and $p'$ is such that for all values of $m \in [m_0, (1 + \eps/6)m_0]$ the graph $G_m$ is a.a.s.\ `in between', that is $G^{-} \subseteq G_m \subseteq G^{+}$. \

Let $G$ denote the $2$-core of $G_m$ and let $H$ be a graph removed by the adversary, as in Definition~\ref{def:complicated_resilience}. We split the graph $G - H$ into a sparse graph $\Gamma'$ which mostly contains edges from $G^{-} - H$ (edges outside of $G^-$ are borrowed to handle atypical vertices), and the rest $G - H - \Gamma'$. Ideally, we would like that $\Gamma'$ is a backbone graph. For $m$ close to $n\log n/6$, the graph $G$ contains short paths between two tiny vertices, some of which exist in $\Gamma'$ as well preventing it from satisfying \ref{P-second-W1}. We circumvent this by constructing a new graph $\Gamma$ from $\Gamma'$ by \emph{contracting} all paths of length at most two between tiny vertices. As tiny vertices do not clump, this has only a mild impact on the structure of the graph and all the other properties remain satisfied. Due to these contractions the graph $\Gamma$ is not a subgraph of $G - H$ any more, however, the subgraph induced by typical vertices which are not a part of contracted paths, is. This allows us to apply Lemma \ref{lem:boosters} to conclude that such $\Gamma$ is either Hamiltonian in which case we are done by `unrolling' the contracted vertices, or contains many boosters between `real' vertices of $G$. Using Lemma \ref{lem:extension}, we subsequently show that many such boosters appear in $G - H$ which allows us to complete a Hamilton cycle.

We point out that by requiring $m$ just a bit larger, that is $m \geq (1 + \eps)n\log n/4$, would be enough for \ref{P-second-W1} to hold already in $\Gamma'$ and the contraction process would not be necessary.
\begin{proof}
  Let $p_0 = (1 - \eps/16) m_0 / \binom{n}{2}$, $p' = (\eps/2) p_0$, and let $G^+$ be the union of independent random graphs $G^- \sim G_{n, p_0}$ and $G_{n, p'}$. Then $G^{+}$ is distributed as $G_{n, p_1}$, where $p_1 = 1 - (1 - p_0)(1 - p')$. By Lemma~\ref{lem:gnp-edge-distribution} we a.a.s.\ have both $e(G^-) \leq m_0$ and $e(G^+) \geq (1 + \eps/6)m_0$.

  Let $\delta_{t}' = \delta_{{\ref{lem:bad-subgraphs}}}(\eps/2)$, $\delta_{a}' = \min\{ \eps/64, \delta_{{\ref{lem:bad-subgraphs}}}(\eps/2) \}$, $c = c_{\ref{lem:gnp-edge-distribution}}$, $L = \max\{ L_{\ref{lem:bad-subgraphs}}(\eps/2), L_{\ref{cl:no-small-ham-deg-tree}}(\eps) \}$, $\mu = \mu_{\ref{lem:extension}}(\eps/4)$, $c' = c'(c, \mu)$ sufficiently large (cf.~\eqref{eq:gq-edge-concentration}), and set the constants given in the statement of the proposition as follows:
  \[
    \delta_{t} = \delta_{t}'/40, \quad \delta_{a} = \max\{ \eps/4, 16\delta_{a}' \}, \quad \text{ and } \quad K = \max\{ 40L, \left(64c' / \eps\right)^{2} \}.
  \]

  For the rest of the proof consider some $m_0 \le m \le (1 + \eps/6)m_0$. Let $G$ be the $2$-core of $G_m$ and set $V = V(G)$. Note that $G$ can be obtained using the following procedure: initially set $G = G_m$, and as long as $G$ contains a vertex of degree at most one, remove it. Let $R$ denote the set of removed vertices. The definition of the procedure then implies that there are at most $|R|$ edges incident to $R$ in $G_m$. By Lemma~\ref{lem:bad-subgraphs} applied to $G^+$ (with $k = 3$), each vertex has at most two neighbours in $\TINY_{p_0, \delta_{t}'}(G^-)$ thus no vertex which is not in $\TINY_{p_0, \delta'_{t}}(G^-)$ can ever be removed. This also implies that the degree of every vertex decreases by at most $2$. As thus $R \subseteq \TINY_{p_0, \delta_{t}'}(G^-)$, Lemma \ref{lem:atyp-size} implies $|R| = O(n / \log^3 n)$.

  Similarly as in the proof of Proposition \ref{prop:fixed-range-pm}, we define $\TINY$ and $\ATYP$ as follows:
  \begin{align*}
    \TINY &:= (\TINY_{p_0, \delta_{t}'}(G^-) \cup \TINY_{p_1, \delta_{t}'}(G^+)) \cap V, \\
    \ATYP &:= (\ATYP_{p_0, \delta_{a}'}(G^-) \cup \ATYP_{p_1, \delta_{a}'}(G^+)) \cap V.
  \end{align*}
  Using the previous observation on the number of removed vertices and edges, it is easy to see that for all $m_0 \leq m \leq (1 + \eps/6)m_0$ we have
  \[
    \TINY_{p, \delta_{t}}(G) \subseteq \TINY \qquad \text{and} \qquad \ATYP_{p, \delta_{a}}(G) \subseteq \ATYP,
  \]
  where $p = e(G) / \binom{v(G)}{2}$. Consider a graph $H \subseteq G$ such that
  \[
    \deg_H(v) \leq
    \begin{cases}
      \deg_{G}(v) - 2, & \text{if } v \in \TINY, \\
      \deg_{G}(v) - K, & \text{if } v \in \ATYP \setminus \TINY, \\
      (1/2 - \eps)np_1, & \text{otherwise}.
    \end{cases}
  \]
  We show that $G - H$ contains a Hamilton cycle, which in turn implies that $G$ is $(1/2 - \eps, \delta_{t}, 2, \delta_{a}, K)$-resilient with respect to Hamiltonicity (as in the case of perfect matchings).

  The key to our proof is the fact that tiny and atypical vertices cannot be clumped up in $G$, captured by the following properties:
  \begin{enumerate}[leftmargin=3.5em, font={\bfseries\itshape}, label=(H\arabic*)]
    \item\label{H-third-neighbourhood} for all $v \in V$ we have $|N^{3}_{G}(v) \cap \TINY| \leq 2$ and $|N^{2}_{G}(v) \cap \ATYP| \leq L$,
    \item\label{H-no-tiny-cycles} every cycle $C \subseteq G$ of size at most $6$ contains at most one vertex from $\TINY$.
  \end{enumerate}
  By Lemma~\ref{lem:bad-subgraphs} applied for $k = 3$ and $\eps/2$ as $\eps$, one easily checks that properties \ref{H-third-neighbourhood} and \ref{H-no-tiny-cycles} hold in $G^+$, and hence in any subgraph $G \subseteq G^+$. Note that we can indeed apply Lemma~\ref{lem:bad-subgraphs} with such parameter, as $p_0 \geq (1 + \eps/2) \log n/(3n)$.

  In the first step of the proof we carefully construct a backbone graph $\Gamma$. Set $q = (\mu/4)p_0$, where $\mu = \mu_{\ref{lem:extension}}(\eps/4)$, and consider a graph $G_q$ obtained by keeping every edge of $G^-$ independently with probability $\mu/4$. Moreover, remove from $G_q$ all vertices which are not in $V$. By the discussion from the beginning of the proof, removing such vertices decreases the degree of every vertex in $G_q$ by at most $2$. We now show that $G_q$ a.a.s.\ has certain properties.

  As $G^{-}$ a.a.s.\ satisfies the assertion of Lemma \ref{lem:gnp-edge-distribution}, we also have that $G_q$ a.a.s.\ satisfies
  \begin{equation} \label{eq:gq-edge-concentration}
    |e_{G_q}(X, Y) - |X||Y|q| \le c' \sqrt{|X||Y|nq},
  \end{equation}
  for every two subsets $X, Y \subseteq V$, for some sufficiently large constant $c'$ depending on $c$ and $\mu$. Next, we claim that most of the vertices in $V \setminus \ATYP$ have degree at least $(1 - 2\delta_{a}')nq$ in $G_q$, as well as degree at most $(1/2 - \eps/2)nq$ in $H_q = G_q \cap H$. Fix $v \in V \setminus \ATYP$. By Chernoff bounds we get
  \[
    \Pr[\deg_{G_q}(v) < (1 - 2\delta_{a}')nq] + \Pr[\deg_{H_q}(v) > (1/2 - \eps/2)nq] \leq e^{-\Omega_{\eps}(nq)} \leq n^{-2\gamma},
  \]
  for some $\gamma = \gamma(\eps) > 0$. Therefore, by Markov's inequality there are at most $n^{1 - \gamma}$ vertices which satisfy at least one of these two conditions. We call all such vertices {\em degenerate} and denote them by $D$. Similarly as atypical vertices, such vertices cannot be clumped in $G$, and thus neither in $G_q$. This is formalised in the following claim.
  \begin{claim}\label{cl:no-small-ham-deg-tree}
    There exists a positive constant $L(\eps)$ such that a.a.s.\ for every $v \in V$ we have $|N_{G}^{2}(v) \cap D| < L$.
  \end{claim}
  \begin{proof}
    The proof goes along the lines of that of Claim~\ref{cl:no-small-pm-deg-tree}. We omit the details.
  \end{proof}

  In conclusion, $G_q$ a.a.s.\ satisfies \eqref{eq:gq-edge-concentration}, every vertex in $V \setminus (\ATYP \cup D)$ has degree at least $(1 - 2\delta_a')nq$ in $G_q$ and at most $(1/2 - \eps/2)nq$ in $H_q$, and no vertex has more than $L$ vertices from $D$ in its second neighbourhood. From now on we choose one such graph $G_q$. In particular, \eqref{eq:gq-edge-concentration} implies $e(G_q) \leq (\mu/4)n^2p_0$.

  We now construct a graph $\Gamma'$ which is `almost' a backbone graph, and then convert it into a true backbone graph, thereby overcoming the issue of two tiny vertices being on a short path, as discussed previously. Let $W_1' := \TINY$, $W_2' := (\ATYP \cup D) \setminus \TINY$, and $U' := V \setminus (W_1' \cup W_2')$. Take $\Gamma'$ to be a graph on the vertex set $V$ obtained by taking all edges in $G_q[U'] - H_q[U']$ and adding some of the edges of $G - H$ incident to vertices in $W_1' \cup W_2'$ such that for every $v \in W_1'$ we have $\deg_{\Gamma'}(v) = 2$, and for every $u \in W_2$ we have $\deg_{\Gamma'}(u) \geq K - 2$ (by requiring that all $v \in W_1$ have $\deg_{\Gamma'}(v) = 2$ we potentially remove at most two incident edges to a vertex $u \in W_2'$). Note that $e(\Gamma') \leq (\mu/2)n^2 p_0$. Recall that property \ref{P-second-W1} does not necessarily hold in $\Gamma'$, thus we cannot claim it is a backbone graph. We take care of this issue as follows: let $\Gamma$ be a graph obtained from $\Gamma'$ by contracting every $uv$-path of length at most two, where $u, v \in W_1'$, and keeping exactly two edges incident to the newly obtained vertex, namely the ones incident to $u$ resp.\ $v$ in $\Gamma'$ that are not a part of the $uv$-path. We also remove all multi-edges and loops. Observe that $\Gamma$ is well-defined since there cannot exist a vertex $w \in V$ which belongs to two such paths $uv$-paths, due to property \ref{H-third-neighbourhood}.

  In order to show that $\Gamma$ is a backbone graph, we first define a witness partition $U \cup W_1 \cup W_2$. Let $X$ be the set of vertices of $\Gamma$ obtained by contractions, $Y$ be the set of all vertices in $ W_2' \cup U'$ that are inner vertices of some contracted path, and $Z$ the set of all vertices of $W_1'$ that are the endpoints of such paths. We now define
  \[
    W_1 := (W_1' \setminus Z) \cup X, \qquad W_2 := W_2' \setminus Y, \qquad \text{and} \qquad U := U' \setminus Y.
  \]
  In other words, $W_1$ consists of all newly formed vertices (obtained by contractions) as well as all remaining vertices of $W_1'$, and $W_2$ and $U$ consist of all vertices in $W_2'$ resp.\ $U'$ that are not a part of any contracted path.

  We are now ready to show that $\Gamma$ is an $(\eps/4, K/10, q)$-backbone graph. Recall that $|V| \ge n - n/\log^2 n$. For every vertex $v \in U \cup W_2$ we have
  \begin{equation}\label{eq:min-deg-contracted}
    \deg_{\Gamma}(v) \geq \deg_{\Gamma'}(v) - 1,
  \end{equation}
  by using property \ref{H-third-neighbourhood}. We now check all the requirements of Definition~\ref{def:backbone}:

  \ref{P-size-bad} follows by Lemma~\ref{lem:atyp-size} and an upper bound on the size of $D$:
  \[
    |W_1 \cup W_2| \leq |\ATYP| + |D| \leq \frac{2n}{\log^3 n} + n^{1 - \gamma} \leq \frac{n}{2 \log^2 n}.
  \]

  \ref{P-degree-bad} follows by construction. Indeed, if for any $v \in W_1$ we have $\deg_{\Gamma}(v) < 2$, then one easily checks that we arrive to a contradiction with either \ref{H-third-neighbourhood} or \ref{H-no-tiny-cycles}. The second part of the property follows from \eqref{eq:min-deg-contracted}.

  \ref{P-second-W1} also holds by construction: if there still exists a path of length at most two between two vertices from $W_1$, then this would contradict~\ref{H-third-neighbourhood}.

  \ref{P-second-neighbourhood} follows from \ref{H-third-neighbourhood}, Claim~\ref{cl:no-small-ham-deg-tree}, and our choice of $K$, since all the edges incident to the vertices in $W_1$ already exist in $\Gamma'$ and as the number of vertices from $W_2$ that are in $N_{G}^{2}(v)$ of any vertex $v$ can at most double in $\Gamma$.

  Lastly, we show the expansion properties of vertices in $U$, i.e.\ property \ref{P-expansion}. By \eqref{eq:min-deg-contracted} and the fact that $N_{\Gamma'}(S) \subseteq N_{\Gamma'}(S')$ whenever $S \subseteq S'$, it suffices to show
  \[
    |N_{\Gamma'}(S)| \ge
    \begin{cases}
      |S| \sqrt{nq} + |S|, &\text{if } |S| < K/(10q),\\
      (1/2 + \eps/8)n + |S|, &\text{if } |S| = \floor{K/(10q)}.
    \end{cases}
  \]
  For every $v \in U$ we have
  \begin{align*}
    \deg_{\Gamma'}(v, U) &= \deg_{G_q}(v) - \deg_{H_q}(v) - \deg_{G_q}(v, W_1 \cup W_2) \\
    &\geq (1 - 2\delta_{a}')nq - (1/2 - \eps/2)nq - K/10 - 2 \\
    &\geq (1/2 + \eps/4)nq,
  \end{align*}
  by \ref{H-third-neighbourhood}, Claim~\ref{cl:no-small-ham-deg-tree}, and our choice of $K$ and $\delta_{a}'$. Take $S \subseteq U$ to be an arbitrary subset of size $|S| \leq K/(10q)$ and let $T := N_{\Gamma'}(S) \cap U$. From the previously obtained bound on $\deg_{\Gamma'}(v, U)$, we have
  \[
    e_{\Gamma'}(S, T) \geq |S|(1/2 + \eps/4)nq.
  \]
  Assume towards a contradiction that $|T| < 2|S|\sqrt{nq}$. Then, from $\Gamma'[U] \subseteq G_q[U]$ and \eqref{eq:gq-edge-concentration} we derive
  \[
    |S| (1/2 + \eps/4) nq \leq e_{\Gamma'}(S, T) \leq |S||T|q + c'\sqrt{|S||T|nq} \leq 2|S|^2 (nq)^{1/2} q + 2c'|S| (nq)^{3/4},
  \]
  which is a contradiction. On the other hand, if $|S| = \floor{K/(10q)}$ then, assuming $|T| < (1/2 + \eps/8)n + \eps n / 16$, again from \eqref{eq:gq-edge-concentration} we have
  \[
    |S| (1/2 + \eps/4)nq \leq e_{\Gamma'}(S, T) \leq |S||T|q + c'\sqrt{|S||T|nq} < |S|(1/2 + 3\eps/16)nq + \eps|S|nq/16,
  \]
  where the last inequality follows from our choice of $K$. We have a contradiction once again.

  To conclude, we obtained an $(\eps/4, K/10, q)$-backbone graph $\Gamma$ with the witness partition $V(\Gamma) = U \cup W_1 \cup W_2$ and at most $(\mu/2) n^2 p_0$ edges, and graphs $G_q$ and $H_q$ applicable by Lemma \ref{lem:boosters}.

  Note that, by the construction of $\Gamma$, the following is now true: if for any subset of edges $E' \subseteq E(G^{-}[U] - H)$ the graph $\Gamma + E'$ is Hamiltonian, then $G - H$ is Hamiltonian as well. Indeed, let $x_{uv} \in X$ be some vertex obtained by contracting a $uv$-path in $\Gamma'$. The two edges incident to $x_{uv}$ in $\Gamma + E'$ necessarily lie on a Hamilton cycle. Moreover, as they correspond to two edges in $\Gamma' \subseteq G - H$, one incident to $u$ and the other to $v$, by splitting every such vertex $x_{uv}$ back into the $uv$-path we obtain a Hamilton cycle in $G - H$.

  The following claim, akin to \cite[Lemma 3.4]{ben2011resilience}, allows us to complete the proof.
  \begin{claim}{\rm\cite{ben2011resilience}}\label{cl:hamiltonicity}
    If for every subset $E' \subseteq E(G^{-}[U] - H)$ of $|E'| \leq n$ edges such that $\Gamma + E'$ is not Hamiltonian there is a vertex $v \in U$ satisfying
    \[
      |N_{G}(v) \cap B_{\Gamma + E'}(v)| > \deg_{H}(v),
    \]
    then $G - H$ is Hamiltonian. \qed
  \end{claim}
  Let $E' \subseteq E(G^{-}[U] - H)$ be a set of edges of size at most $|E'| \leq n$. It is not too difficult to see that $\Gamma + E'$ is an $(\eps/4, K/4, q)$-backbone graph. Indeed, none of the properties \ref{P-size-bad}--\ref{P-second-W1} can be violated by adding edges with both endpoints in $U$. Similarly, the expansion property \ref{P-expansion} is not affected by addition of edges. Lastly, property \ref{P-second-neighbourhood} holds as $\Gamma + E' \subseteq G^{-}[U] - H$ and by referring to \ref{H-third-neighbourhood}. By Lemma~\ref{lem:boosters} applied with $G_q$ (for $\eps/4$ as $\alpha$) we get that the set of vertices $v \in U$ such that $|B_{\Gamma + E'}(v)| \geq (1/2 + \eps/4)n$ is of size at least $(1/2 + \eps/4)n$. As $e(\Gamma + E') \leq \mu n^2p$, by Lemma~\ref{lem:extension} applied with $\eps/4$ as $\alpha$ and $G^{-}$ as $G$, there exists a vertex $v \in U$ with $|N_{G^{-}}(v) \cap B_{\Gamma + E'}(v)| > np_0/2 > \deg_{H}(v)$. Finally, Claim~\ref{cl:hamiltonicity} implies that the graph $G - H$ is Hamiltonian.
\end{proof}

Having Proposition \ref{prop:fixed-range-ham}, the proof of the following theorem and Theorem \ref{thm:main-thm-ham} are identical to the proofs of Theorem \ref{thm:stronger-resilience-pm} and Theorem \ref{thm:main-thm-pm}, using \cite[Theorem 1.1]{LS12} instead of \cite[Theorem 3.1]{SV08} to handle $m \ge C n \log n$.
\begin{theorem}\label{thm:stronger-resilience-ham}
  Let $\eps > 0$ be a constant and consider the random graph process $\{G_i\}$. There exist positive constants $\delta_{t}(\eps)$, $\delta_{a}(\eps)$, and $K(\eps)$ such that a.a.s.\ for every $m \geq \frac{1 + \eps}{6} n\log n$ we have that the $2$-core of $G_m$ is $(1/2 - \eps, \delta_t, 2, \delta_a, K)$-resilient with respect to being Hamiltonian. \qed
\end{theorem}

\subsection{Proof of Lemma~\ref{lem:boosters} and Lemma~\ref{lem:extension}}\label{sec:rotation-extension}

Let us first give a brief outline of how to apply P\'{o}sa's rotation-extension technique to backbone graphs with the goal of constructing long paths.

Assume $\Gamma$ is a backbone graph and let $P = v_0, \ldots, v_\ell$ be a path in $\Gamma$. If $\{v_0, v_\ell\}$ is an edge in $\Gamma$ then such a path can be closed into a cycle. If the obtained cycle does not cover all vertices, then one easily checks that properties \ref{P-degree-bad}--\ref{P-expansion} imply connectivity of $\Gamma$ and we can thus {\em extend} $P$ into a path longer than $P$.

Suppose that $P$ cannot be extended and let $\{v_0, v_i\}$ be an edge in $\Gamma$ for some $2 \leq i \leq \ell - 1$. Then the path $P' = v_{i - 1}, \ldots, v_0, v_i, \ldots, v_\ell$ is another path in $\Gamma$ of the same length $\ell$. We say that $P'$ is obtained from $P$ by a {\em rotation} around the {\em endpoint} $v_0$, with {\em pivot point} $v_i$, and {\em broken edge} $\{v_{i - 1}, v_i\}$ (see, Figure~\ref{fig:rotation}). Observe that by performing a rotation we can now possibly obtain a cycle by adding the edge $\{v_{i - 1}, v_\ell\}$ as well. Otherwise, we perform more rotations to obtain more boosters. The rotation is repeated until we find a closing edge in $\Gamma$.
\begin{figure}[!htbp]
  \centering
  \begin{subfigure}{0.47\textwidth}
    \centering
    \includegraphics[scale=0.625]{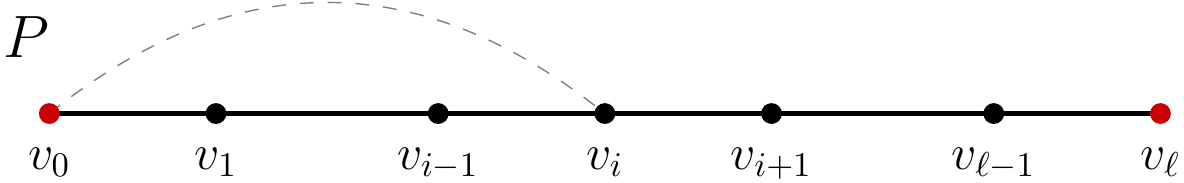}
  \end{subfigure}
  \hspace{1em}
  \begin{subfigure}{0.47\textwidth}
    \centering
    \includegraphics[scale=0.625]{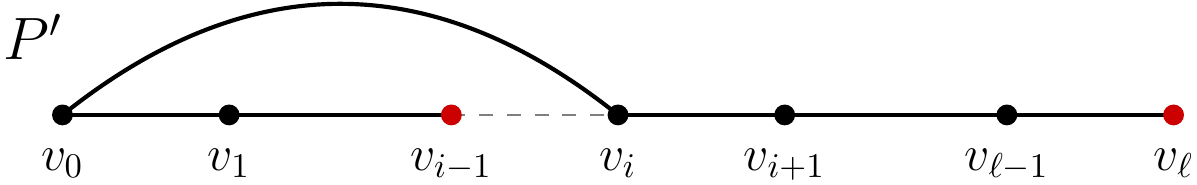}
  \end{subfigure}
  \caption{Rotation of the the path $P$ around the fixed endpoint $v_0$, with pivot point $v_i$, and the broken edge $\{v_{i - 1}, v_i\}$. Pairs of red vertices denote the endpoints of the paths, i.e.\ boosters.}
  \label{fig:rotation}
\end{figure}

In our setting, we have to be careful with vertices of low degree. This is why we need properties \ref{P-degree-bad}--\ref{P-second-neighbourhood}. We illustrate in particular why we need
property \ref{P-second-W1}. Suppose $P = v_0, \ldots, v_\ell$ is a longest path in $\Gamma$ such that $v_0, v_{i - 1} \in W_1$ for some $v_{i}$ which is the only neighbour (other than $v_1$) of $v_0$ on the path $P$. Rotating around $v_0$ with pivot point $v_{i}$ yields a path $P' = v_{i - 1}, \ldots, v_0, v_{i}, \ldots, v_\ell$. Now, the only rotation that can be performed around $v_{i - 1}$ brings the path $P'$ back to where we started from. In conclusion, in such a situation we cannot prove that $\Gamma$ contains many boosters. Property \ref{P-second-W1} guarantees that this cannot occur.
\begin{proof}[Proof of Lemma~\ref{lem:boosters}]
  Let $P = v_0, \ldots, v_\ell$ be a longest path in $\Gamma$. For a subset of vertices $Z \subseteq V(P)$ we write $Z^{+} := \{v_{i + 1} \colon v_i \in Z \}$ and $Z^{-} := \{v_{i - 1} \colon v_i \in Z\}$. For a vertex $z$ we abbreviate $\{z\}^{+}$ to $z^{+}$ and $\{z\}^{-}$ to $z^{-}$.

  We first show that there exists a longest path with both endpoints in $U$. Suppose $v_0 \in W_2$. Then by \ref{P-degree-bad} and \ref{P-second-neighbourhood} we know that there are at least $K$ neighbours $u \in N_{\Gamma}(v)$ on the path $P$ which have $u^{-} \in U$. Thus, by performing a single rotation around $v_0$ with pivot point $u$ we get a path $P'$ with an endpoint in $U$. Similarly, if $v_0 \in W_1$ its only other neighbour on the path $u$ has $u^{-}$ belonging to either $U$ or $W_2$, by \ref{P-second-W1}. In any case, by performing at most two rotations we reach a path $P'$ with an endpoint in $U$. Repeating the same argument for $v_\ell$ yields the claimed path $P^*$ with both endpoints in $U$.

  \textbf{Phase 1: Initial rotations}

  The first phase consists of repeatedly rotating the first endpoint as in \cite[Lemma 3.2~{\bf Step 1}]{LS12} in order to obtain a set of at least $n/4$ new endpoints of paths of length $\ell$ that all contain the same vertices and all end in $v_\ell$. Compared to \cite{LS12}, the only difference in our argument is that we ignore all pivot points $u$ with at least one of $u$, $u^{+}$, and $u^{-}$ not belonging to $U$. As every vertex can have at most $K$ such neighbours $u$, one easily sees that all the calculations of \cite{LS12} essentially remain the same. Indeed, starting with $X_0 = \{v_0\}$ and denoting by $X_i$ the set of endpoints obtained by exactly $i$ rotations, we get
  \[
    |X_{i + 1}| \geq \frac{1}{2} \Big( |N_{\Gamma}(X_{i})| - K|X_{i}| - 3 \sum_{j = 0}^{i} |X_i| \Big).
  \]
  Using identical calculations as in \cite{LS12} yields $|X_{i + 1}| \geq (nq / 4)^{(i + 1)/2}$, for all $i \geq 0$ with $|X_{i}| \leq K/q$. Once we reach a set $|X_m| = \max\{1, \floor{K/q}\}$, which is easily observed to happen after at most $O(\log n/\log\log n)$ many steps, the above bound on the size of $X_{i + 1}$, together with property \ref{P-expansion}, immediately implies $|X_{m + 1}| \geq n/4$.

  \textbf{Phase 2: Terminal rotation}

  At the end of the first phase we have a set $X$ of at least $n/4$ possible endpoints. Let $Y$ denote the set of endpoints that can be generated by exactly one more rotation starting from $X$.

  In \cite[Lemma 3.2~{\bf Step 2}]{LS12} it is shown that by partitioning the path $P$ into appropriately many intervals, denoted by $P_i$, one can define pairs of vertex sets $(X_i, Y_i)$ such that there is no edge between $X_i$ and $Y_i$ in $G_q - H_q$, for every $i$. From Lemma~\ref{lem:gnp-edge-distribution} we know lower bounds on the number of edges in $E_{G_q}(X_i,Y_i)$ and thus all these need to belong to $H_q$. In \cite[Lemma 3.2~{\bf Step 2}]{LS12} it is shown that $|Y| < (1/2 + \alpha)n$ implies a contradiction to the degree assumption of $H_q$. The only difference in our case is that we again need to ignore all pivot points $v^{+} \in P_i$ with $v \notin U$ (which should belong to $Y_i$ now), and similarly $v^{-}$. However, as every interval $P_i$ (in line 4 of the proof of {\bf Step 2}) is of size at least $n (\log\log n)^{1/2}/(4 \log n)$ and by \ref{P-size-bad} each interval $P_i$ has at most $o(|P_i|)$ such `bad' vertices, we can simply ignore them since this contributes $o(n^2q)$ many edges. All remaining calculations from \cite[Lemma 3.2]{LS12} remain the same and we obtain a set of new endpoints of size at least $(1/2 + \alpha)n$.

  \textbf{Phase 3: Rotating the other endpoint}

  Exactly as in \cite[Lemma 3.2~{\bf Step 3}]{LS12} for every of the newly obtained $(1/2 + \alpha)n$ endpoints, we analogously rotate the other endpoint to obtain the intended result.
\end{proof}

We wrap-up by giving a proof of Lemma~\ref{lem:extension}.
\begin{proof}[Proof of Lemma~\ref{lem:extension}]
  Consider a graph $\Gamma$ with at most $\mu n^2 p$ edges and a subset $U \subseteq V(G)$ such that $\Gamma[U] \subseteq G$ and
  \[
    A := \{ v \in U \colon |B_{\Gamma}(v) \cap U| \geq (1/2 + \alpha)n \}
  \]
  is of size $|A| \geq (1/2 + \alpha)n$. Take $A' \subseteq A$ to be a set of size $|A'| = \alpha n/2$, and for every $v \in A'$ let $B'(v) := (B_{\Gamma'}(v) \cap U) \setminus A'$ and note that $|B'(v)| \geq (1/2 + \alpha/2)n$. By Chernoff bounds we have that the probability of a fixed vertex $v \in A'$ having fewer than $np/2$ neighbours inside of the set $B'(v)$ in the graph $G$ is at most $e^{-\Omega_{\alpha}(np)}$. Since $A'$ is disjoint from all the sets $B'(v)$, these events are independent for different vertices $v, u \in A'$. Therefore, the probability that no vertex $v \in A'$ has more than $np/2$ neighbours inside its respective set $B'(v)$ in $G$ is bounded by $e^{-\Omega_{\alpha}(n^2p)}$. Hence, the probability of failure is upper bounded by
  \[
    \Pr \leq e^{-\Omega_{\alpha}(n^2 p)} \cdot \sum_{e(\Gamma') \leq \mu n^2 p} \Pr[\Gamma' \subseteq G].
  \]
  The probability for a fixed graph with $t$ edges to be a subgraph of $G$ is $p^{t}$ and thus by a union bound over all such graphs, we get
  \[
    \Pr \leq e^{-\Omega_{\alpha}(n^2 p)} \cdot \sum_{t = 1}^{\mu n^2 p} \binom{\binom{n}{2}}{t} p^{t} \leq e^{-\Omega_{\alpha}(n^2 p)} \sum_{t = 1}^{\mu n^2p} \left( \frac{e n^2p}{t} \right)^{t}.
  \]
  One easily sees that for $0 < \mu \leq 1$ the right hand side is increasing for $1 \leq t \leq \mu n^2p$ and hence we may substitute $t = \mu n^2p$ to conclude
  \[
    \Pr \leq e^{-\Omega_{\alpha}(n^2 p)} \cdot (\mu n^2 p) \cdot \left( \frac{e}{\mu} \right)^{\mu n^2p} = e^{-\Omega_{\alpha}(n^2 p)} \cdot e^{O(\mu \log{(1/\mu)} n^2 p)} = o(1),
  \]
  for sufficiently small $\mu$ depending on $\alpha$.
\end{proof}

%!TEX root = threshold_resilience.tex
\paragraph*{Acknowledgements.} The third author would like to thank Michael
Krivelevich for directing his attention to \cite{krivelevich2015long} which
helped in making some of the arguments cleaner.

\bibliographystyle{abbrv}
\bibliography{references}

\end{document}